\documentclass[ 11pt]{article}
\usepackage{amsmath}
\usepackage{amssymb}
\usepackage{tipa}
\usepackage{textcomp}
\usepackage{amsthm}
\usepackage{array}
\usepackage{color}
\usepackage{fancyhdr}
\usepackage{enumerate}
\usepackage{multicol}
\usepackage{titlesec}

\usepackage{cancel}
\usepackage{amsthm}
\usepackage{pgfplots}

\theoremstyle{plain}

\newtheorem{corollary}{Corollary}
\newtheorem{lemma}{Lemma}

\theoremstyle{definition}
\newtheorem{definition}{Definition}

\newtheorem{thm}{Theorem}
\theoremstyle{remark}

\begin{document}

\title{A New Strategy for Riemann Hypothesis}
\author{Irfan Okay}
\maketitle

\begin{abstract}
A square-free integer is a positive integer that is not divisible by the square of any prime. Merten's function, $M(x)$ is defined as the difference between the number of square free integers with an even number of prime factors and the number of those with an odd number of prime factors up to $x$.  It's well known that  $M(x)=O(x^{1/2+\epsilon})$ is equivalent to Riemann Hypothesis. In this paper, we derive a new equation involving Merten's function that leads to a promising technique for Riemann Hypothesis.
\end{abstract}

\section{Introduction}
A square free integer is a positive integer that is not divisible by the square of any prime number.  Merten's function is defined by

\begin{equation}
M(x)=\sum_{a\leq x} \mu(a)
\end{equation}

where $\mu$ is the Mobius function;

\begin{equation} \mu(a)=\begin{cases}
                   0, & \mbox{if } p^2/a\quad  \mbox{for some p} \\
                   1, & \mbox{if } a=1 \\
                   (-1)^{\omega{(a)}}, & \mbox{otherwise } \\
                 \end{cases}
\label{mu}\end{equation}

where  $\omega(a)$ is the number of prime factors of $a$. Basically, Merten's function computes the difference between the number of square free integers with an even number of prime factors, and the number of those with an odd number of primes factors. It is well known that Riemann Hypothesis is equivalent to

\begin{equation}
M(x)=O(x^{1/2+\epsilon})  \label{mert}
\end{equation}

given $\epsilon >0$, for all sufficiently large $x$ \cite{littlewood1}. A stronger conjecture, known as Merten's conjecture is
\begin{equation}
M(x)=O(x^{1/2})  \label{mertt}
\end{equation}

which was disproved \cite{Mertens1}, \cite{Mertens2}. In this paper, we prove the following equation;

\begin{equation}\sum_{n=1}^{m} \sum_{i=1}^{p-1} M'_{p}(\frac{x}{np},  \quad \frac{x}{np-i})
=-([\log_px])(p-1) \label{result}
\end{equation}

for any positive irrational number $x$ and any prime $p<x$. Here  $m=[x/p]+2$. $[.]$ denotes the integer part of the number inside the bracket and  $M_p'(a, y)$ is the Mobius sum taken over the square free integers in the interval $(x, y)$ that are not divisible by $p$;

\begin{eqnarray}
M'_{p}(x, \quad y)&=&\sum_{x<a< y, p\nmid a} \mu(a)\label{defmp'}
\end{eqnarray}

Note that $M(0, x)=M(x),  M_p'(0, x)=M_p'(x)$, and we have $M(x)=M_p'(x)+M_p'(x/p)$.   Thus, any bound on $M_p'(x)$ can easily be converted to a bound on $M(x)$.  The irrationality assumption was only needed to ensure that none of the boundary points in the sum above will contain an integer. 

\section{A Strategy for Riemann Hypothesis}

Here is a rough visual sketch of the intervals in the double sum in (\ref{result}) for $p=7$.
\bigskip
\bigskip
\bigskip

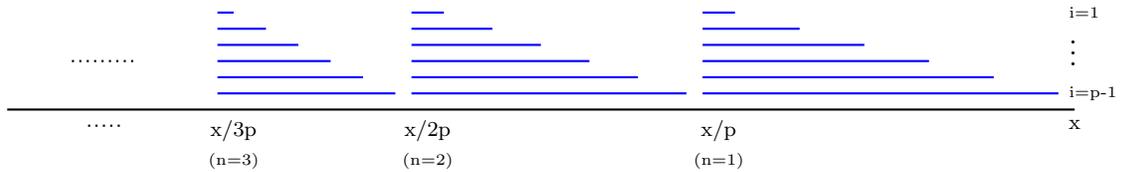
\begin{tikzpicture}[thick, scale=0.43]

 \draw [thick, black, dotted] (-33,1.5) -- (-31,1.5);

 \draw [thick, black, dotted] (-32.5,-0.5) -- (-31.5,-0.5);

 \draw [thick, black, -] (-35,0) -- (-2,0);
     \draw [thick, blue, -] (-28.5,0.5) -- (-23,0.5);
     \draw [thick, blue, -] (-28.5,1) -- (-24,1);
     \draw [thick, blue, -] (-28.5,1.5) -- (-25,1.5);
     \draw [thick, blue, -] (-28.5,2) -- (-26,2);
     \draw [thick, blue, -] (-28.5,2.5) -- (-27,2.5);
     \draw [thick, blue, -] (-28.5,3) -- (-28,3);

     \node(-22.5, 0) {.};
      \node[below] at (-28, 0){\scriptsize {x/3p}};
     \node[below] at (-22, 0){\scriptsize {x/2p}};
     \node[below] at (-13, 0) {\scriptsize x/p};
     \node[below] at (-2, 0) {\scriptsize x};

      \node[below] at (-28, -1) {\tiny(n=3)};
       \node[below] at (-22, -1) {\tiny(n=2)};
     \node[below] at (-13, -1) {\tiny(n=1)};

       \node[right] at (-2.5, 3) {\tiny i=1};
       \node[right] at (-2.5, 0.5) {\tiny i=p-1};
    \node[right] at (-2.5, 2) {\vdots};

        \draw [thick, blue, -] (-22.5,0.5) -- (-14,0.5);
     \draw [thick, blue, -] (-22.5,1) -- (-15.5,1);
     \draw [thick, blue, -] (-22.5,1.5) -- (-17,1.5);
     \draw [thick, blue, -] (-22.5,2) -- (-18.5,2);
     \draw [thick, blue, -] (-22.5,2.5) -- (-20,2.5);
     \draw [thick, blue, -] (-22.5,3) -- (-21.5,3);

           \draw [thick, blue, -] (-13.5,0.5) -- (-2.5,0.5);
     \draw [thick, blue, -] (-13.5,1) -- (-4.5,1);
     \draw [thick, blue, -] (-13.5,1.5) -- (-6.5,1.5);
     \draw [thick, blue, -] (-13.5,2) -- (-8.5,2);
     \draw [thick, blue, -] (-13.5,2.5) -- (-10.5, 2.5);
     \draw [thick, blue, -] (-13.5,3) -- (-12.5,3);

\end{tikzpicture}

\bigskip
As $p$ and $x$ grows, the size of gaps between the intervals in the bottom rows converge to zero. Also, by using a second series with a smaller but appropriately chosen prime $q$, and subtracting the two series from top rows down, we can largely eliminate the intervals in the upper rows. After the procedure,  we will be left with a number of approximately full intervals. For example, if we take $0<q<p<x<q^2<p^2$, then the exact equation will be;

\[(p-q)M_p'(x)+M_p'(\bigcup G)+M_p'(\bigcup L)=-(p-q)\]

where $G$ is the set of all interval gaps in the bottom $p-q$ rows, and $L$ is the set of all leftover intervals in the top $q$ rows. This gives us

\[M_p'(x)=-\frac{1}{p-q}\left(M_p'(\bigcup G)+M_p'(\bigcup L)\right)-1\]

 Then the task is to show that for any $\epsilon >0$, there is an $N$ such that for all  $x\geq N$, there exists primes $p_x$ and $q_x$ such that
 
\[\frac{1}{p_x-q_x}\left(M_{p_x}'(\bigcup G)+M_{p_x}'(\bigcup L)\right)\leq x^{1/2+\epsilon}\]

\section{The Sketch of Proof}

We first define $S_x(u, v)$ ;
\begin{eqnarray}
S_x(u, \quad v)=\{a| u<a<v, \quad \mbox{and }\quad s(a)>x \}
\end{eqnarray}

where $s(a)$ is the smallest prime divisor of $a$. So, $S_x(u, v)$ is the set of integers that are left in the interval $(u, v)$ after we remove the multiples of all primes up to $x$. Then we prove that(Lemma 1)

\[|S_x(kN_p-x, \quad kN_p)|=[\log_px]+1\]

where \begin{equation}
N_p=\prod_{p_i<x, p_i\neq p}p_i
\end{equation}

Since $S_x(0, x)=1$, this gives us

\begin{equation}
|S_x(kN_p-x, \quad kN_p)|-|S_x(0, \quad x)|=[\log_xp] \label{corollary1}
\end{equation}

Then we obtain a second expression for $|S_x(kN_p-x, \quad kN_p)|-|S_x(0, \quad x)|$ using the classical sieve(Lemma 2);

\begin{equation}
|S_x\left(kN_p-x,\quad kN_p\right)|-|S_x\left(0,\quad  x\right)|=\sum_{n=0}^{\pi(x)}(-1)^{k}\sum_{a\in A_k} {\left[R(0, \dfrac{x}{a})- R(\dfrac{kN_p-x}{a}, \dfrac{kN_p}{a})\right]}\label{lem2}
\end{equation}

where $A_k, k\geq 1$ is the set all of square free integers with k-prime factors whose largest prime factor is less than $x$, with $A_0=\{1\}$, and $R(a, b)$  is the residuals defined as
\begin{eqnarray}
 R (a, b)&=&(b-a)-N(a, b)
 \end{eqnarray}

with $N(a, b)$ being the number of integers in $(a, b)$. Combining two expressions above we get

\begin{equation}
\sum_{n=0}^{\pi(x)}(-1)^{k}\sum_{a\in A_k} {\left[R(0, \dfrac{x}{a})- R(\dfrac{kN_p-x}{a}, \dfrac{kN_p}{a})\right]}
=[\log_xp]
\label{lem2}
\end{equation}

Now, one of the biggest challenges of  sieve methods is that the number of residual terms grow exponentially, with no useful pattern between them. However, in our case above, we show that $ R(0, \dfrac{x}{a})- R(\dfrac{kN_p-x}{a}, \dfrac{kN_p}{a})$ is either $1$ or $0$ (Lemma 3) which converts the messy series above into a counting function on square free integers. Summing up these series for $k=1 \cdots p-1$ gives (\ref{result})(Theorem)

\section{Lemma 1}

For the rest of the paper we will assume that $x$ is a positive irrational number. Thus we won't have to worry about $x/a$ being equal to an integer when  $a$ is an integer.  Given a positive integer $a$, let $s(a)$ denote the smallest prime divisor of $a$. Let $S_x(u, v)$ be the set of integers between $u$ and $v$ that has no prime divisor less than $x$. That is;

\begin{eqnarray}
S_x(u, \quad v)=\{a| u<a<v, \quad \mbox{and }\quad s(a)>x \}
\end{eqnarray}

In other words, $S_x(u, v)$ is the set of integers that are left in $(u, v)$ after we remove all multiples of primes up to $x$. Thus, for example, we have

\begin{equation}
|S_{\sqrt{x}}(0, x)|= \pi(x)-\pi(\sqrt{x})+1
\end{equation}
 and
\begin{equation}
|S_{x}(0, x)|=1
\end{equation}

where $\pi(x)$ is the number of primes less than $x$, the usual prime counting function.

\begin{definition}
Given a prime $p<x$, let $N_p$ be the product of all primes less than $x$ except for $p$,
\end{definition}

\begin{equation}
N_p=\prod_{p_i<x, p_i\neq p}p_i
\end{equation}

\bigskip

We have;

\begin{lemma}
\begin{equation}
|S_x(kN_p-x, \quad kN_p)|=[\log_px]+1
\label{lem1}\end{equation}
for any $k\leq p-1$.
\end{lemma}

\begin{proof}

Let $a\in (kN_p-x, \quad kN_p)$, an integer. Then $a=kN_p-b$ for some $0<b<x$.

 Suppose $b$ has a prime divisor different than $p$, say $c$. Then $0<c<x$, and $c/kN_p$ by definition of $N_p$, which implies $c/(kN_p-b)=a$. Thus
$a\notin S_x(kN_p-x, \quad kN_p)$ by definition of $S_x$.

Now suppose $b$  doesn't have a prime divisor different than $p$. Then $b=1$ or $b=p^n$ for some $n$. But then $a=kN_p-p^k$, is not divisible by any prime $<x$, which implies  $a\in S_x(kN_p-x, \quad kN_p)$.

Thus, the elements of $S_x(kN_p-x, \quad kN_p)$ are precisely those $a$'s where $a=kN_p-p^n$ for some n with $p^n<x$. Thus,
\begin{equation}
S_x(kN_p-x, \quad kN_p) =\{kN_p-1,\quad  kN_p-p, \quad kN_p-p^2\quad \cdots \quad kN_p-p^n\}
\end{equation}
where $p^n<x$; and $ |S_x(kN_p-x, \quad kN_p)|=[\log_px]+1 $
\end{proof}
\bigskip

\begin{corollary}
\begin{equation}
|S_x(kN_p-x, \quad kN_p)|-|S_x(0, \quad x)|=\log_xp \label{corollary1}
\end{equation}
\end{corollary}
\bigskip

\section{Lemma 2}

The difference $|S_x(kN_p-x, \quad kN_p)|-|S_x(0, \quad x)|$ is key to our analysis. As we will see, it has some very useful properties that simplifies the sieving procedure a great deal.  We will now use  the sieve of Eratosthenes to obtain an expression  for  $|S_x(0, x)|$ and $|S_x(kN_p-x, \quad kN_p)|$ in terms of square free integers up to $x$. In the classical sieve, we first remove all multiples of $2$, then remove all multiples of $3$ from the remaining integers,and so on, up to the last prime less than $x$. When we account for all possible overlappings between different sequences of prime multiples, this procedure is equivalent to the following  \\

$|S_x(0, x)|$= (The number of integers in $(0, x)$) \\
\hspace{1in}$-$ (the number of multiples of 2, the number of multiples  3, $ \cdots$ in $(0, x)$) \\
$ +$ (the number of multiples of $2\cdot 3$, the number of multiples $2\cdot 5$, $\cdots$ in $(0, x)$) \\
$-$ (the number of multiples of $2\cdot 3\cdot 5$, the number of multiples of $2\cdot 5\cdot 7$, $\cdots$ in $(0, x)$) $\cdots$\\

We continue until we consume all square free products of all primes less than $x$, alternating signs. To calculate  $|S_x(kN_p-x, \quad kN_p)|$ this way, we just count the multiples of these products in the interval $(kN_p-x, \quad kN_p)$ instead of $(0, x)$. We will now formalize this procedure.

Let $N(a, b)$ be the number of integers between $a$ and $b$, and  $N_m(a, b)$ be the number of multiples of an integer $m$ in $(a, b)$. We have

 \begin{equation}
N_m(a, b)=N(\frac{a}{m}, \frac{b}{m})\label{nm}
\end{equation}

Define  $R(a, b)$  as follows;

\begin{eqnarray}
 R (a, b)&=&(b-a)-N(a, b)
 \end{eqnarray}

 Note that in case of $a=0$, $R(0, b)$ represents the usual residual part of $b$, thus, a positif number between $0$ and $1$. However, for a general interval $(a, b)$, it can be negative as well. Take the interval $(2.8, 3.1)$ for example. We have $N(2.8, 3.1)=1$, but  $R(2.8, 3.1)=(3.1-2.8)-N(2.8, 3.1)=0.3-1=-0.7$

Let $A_n, n\geq 1$ be the set all of square free integers with n-prime factors whose largest prime factor is less than $x$, and define $A_0=\{1\}$.  That is,
\begin{eqnarray*}
A_o &=& \{1\} \\
  A_1 &=& \{p_i |\quad  p<x, prime\} \\
  A_2 &=& \{p_i\cdot p_j| p_i, p_j< x, primes\} \\
  &\vdots&
\end{eqnarray*}
Let $m(a)$ denote the largest prime divisor of $a$.. Then we can write
 \begin{equation}
A_n=\{a|\mu(a)\neq 0, \omega(a)=n, m(a)<x\}\\
\end{equation}

\bigskip
We can express the sieving procedure above as follows;
\begin{lemma}

\begin{equation}
|S_x\left(kN_p-x,\quad kN_p\right)|-|S_x\left(0,\quad  x\right)|=\sum_{n=0}^{\pi(x)}(-1)^{k}\sum_{a\in A_k} {\left[R(0, \dfrac{x}{a})- R(\dfrac{kN_p-x}{a}, \dfrac{kN_p}{a})\right]}\label{lem2}
\end{equation}
For any integer $k\leq p-1$,
\end{lemma}

\begin{proof}
By (\ref{nm}), the sieving procedure above translates into

\begin{eqnarray}
|S\left(0,\quad  x\right)|&=&N(0, x)-\sum_{a\in A_1} {N_a(0, x)}+\sum_{a\in A_2} {N_a(0, x)}-\sum_{a\in A_3} {N_a(0, x)}+\cdots \\
&=&N(0, x)-\sum_{a\in A_1} {N(0, \dfrac{x}{a})}+\sum_{a\in A_2} {N(0, \dfrac{x}{a})}-\sum_{a\in A_3} {N(0, \dfrac{x}{a})}+\cdots \\
&=&\sum_{n=0}^{\pi(x)}(-1)^{k}\sum_{a\in A_k} {N(0, \dfrac{x}{a})} \label{S1}
\end{eqnarray}

Similarly,

\begin{eqnarray}|S\left(kN_p-x,\quad  kN_p\right)|&=&N\left(kN_p-x,\quad  kN_p\right)-\sum_{a\in A_1} {N(\dfrac{kN_p-x}{a}, \dfrac{kN_p}{a})}+\sum_{a\in A_2} {N(\dfrac{kN_p-x}{a}, \dfrac{kN_p}{a})}-\cdots \\ &=& \sum_{n=0}^{\pi(x)}(-1)^{k} \sum_{a\in A_k} {N(\dfrac{kN_p-x}{a}, \dfrac{kN_p}{a})}\\\label{S2}\end{eqnarray}

Thus,
\begin{eqnarray}
|S\left(kN_p-x,  kN_p\right)|-|S(0, x)|
&=&\sum_{n=0}^{\pi(x)}(-1)^{k} \sum_{a\in A_k} \left[ N(\dfrac{kN_p-x}{a}, \dfrac{kN_p}{a})- {N(0, \dfrac{x}{a})}\right]  \label{lem2p1}
\end{eqnarray}

On the other hand, by definition,
\[ N(0, \frac{x}{a})=\frac{x}{a}-R(0, \frac{x}{a}) \]
and
\begin{eqnarray*}
N(\dfrac{kN_p-x}{a}, \dfrac{kN_p}{a})&=&\left(\frac{kN_p}{a}-\frac{kN_p-x}{a}\right)-R(\dfrac{kN_p-x}{a}, \dfrac{kN_p}{a})\\\\
&=&\frac{x}{a}-R(\dfrac{kN_p-x}{a}, \dfrac{kN_p}{a})\\
\end{eqnarray*}

Therefore,
\begin{eqnarray}
N(\dfrac{kN_p-x}{a}, \dfrac{kN_p}{a})- {N(0, \dfrac{x}{a})}&=&\frac{x}{a}-R(\dfrac{kN_p-x}{a}, \dfrac{kN_p}{a})-
\left(\frac{x}{a}-R(0, \frac{x}{a})\right)\\&=& R(0, \dfrac{x}{a})-R(\dfrac{kN_p-x}{a}, \dfrac{kN_p}{a}) \label{nr}
\end{eqnarray}

Substituting this  into (\ref{lem2p1}) gives us (\ref{lem2}).

\end{proof}
\bigskip

\section{Lemma 3}
Here is the key observation about the residuals in (\ref{lem2}).

\begin{lemma}
Given a square free integer $a$, irrational $x$, prime $p<x$ and $k\leq p-1$, we have

\[ R(0, \dfrac{x}{a})-R(\dfrac{kN_p-x}{a}, \dfrac{kN_p}{a})=\begin{cases}
       1, \quad  & if \quad p/a \quad and \quad  R(0, \dfrac{x}{a})> R(0, \dfrac{kN_p}{a}) \\\\
      0, \quad &\mbox{if}\quad  p/a \quad and \quad   R(0, \dfrac{x}{a})\leq R(0, \dfrac{kN_p}{a}) \\\\
      0, \quad & if \quad p\nmid a\\\\
   \end{cases}
\]\label{lem3}
\end{lemma}

\bigskip

\begin{proof}

\textbf{Case 1: }Suppose $p/ a$ and $ R(0, \frac{x}{a})> R(0, \frac{kN_p}{a})$. Then  $p/a$ implies  $\frac{kN_p}{a}$ is not an integer since $p\nmid N_p$, and $k\leq p-1$.  Let
\[\frac{x}{a}=n_1+r_1  \quad \mbox{and}\quad  \frac{kN_p}{a}=n_2+r_2 \]

where $0< r_1, r_2<1$. We have $r_2\neq 0$ by the assumption, and  $r_1\neq 0$ since $x$ is irrational.  Note that

\[N(0, \frac{x}{a})=n_1, R(0, \frac{x}{a})=r_1 \quad \mbox{and}\quad N(0, \frac{kN_p}{a}) =n_2, R(0, \frac{kN_p}{a})=r_2\]

Thus $ R(0, \frac{x}{a})> R(0, \frac{kN_p}{a})$ implies   $r_1> r_2$ . Then,

\begin{eqnarray*}
 N (\dfrac{kN-x}{a}, \dfrac{kN}{a}) &=& N(n_2-n_1+r_2-r_1, \quad n_2+r_2)
\end{eqnarray*}

The integers in the interval $(n_2-n_1+r_2-r_1, \quad n_2+r_2) $ are given by
\[\{n_2-n_1, \quad n_2-n_1+1, \quad n_2-n_1+2, \quad \cdots n_2-1, \quad n_2 \}\]

since $r_2-r_1<0$. The number of terms in the list above is

\[n_2-(n_2-n_1)+1=n_1+1\]

Thus,

\begin{eqnarray*}
 N(\dfrac{kN-x}{a}, \dfrac{kN}{a})&=&n_1+1\\&=&N(0, \frac{x}{a})+1
\end{eqnarray*}

which implies, by (\ref{nr}),

\begin{eqnarray*}
 R(0, \frac{x}{a})-R(\dfrac{kN-x}{a}, \dfrac{kN}{a})&=& N(\dfrac{kN-x}{a}, \dfrac{kN}{a})-N(0, \dfrac{x}{a})\\ \\& = &1\\
\end{eqnarray*}

\textbf{Case 2.} Now suppose $p/ a$ but $R(0, \frac{x}{a}) \leq R(0, \frac{kN_p}{a})$. Let \[\frac{x}{a}=n_1+r_1 \quad  \mbox{and} \quad \frac{kN_p}{a}=n_2+r_2\].

for integers $n_1, n_2$ and $0< r_1, r_2<1$. We have $r_1\leq r_2$ by the assumption. \\

Then
\begin{eqnarray*}
  N(\dfrac{kN-x}{a}, \dfrac{kN}{a}) &=& N(n_2-n_1+r_2-r_1, \quad n_2+r_2)
\end{eqnarray*}

The integers in the interval above are given by

\[\{n_2-n_1+1, \quad n_2-n_1+1, \quad n_2-n_1+2, \quad \cdots n_2-1, \quad n_2 \}\]

Since $0\leq  r_2-r_1<1$ . The number of integers in the list is $n_2-(n_2-n_1+1)+1=n_1$

Thus,

\begin{eqnarray*}
 N(\dfrac{kN-x}{a}, \dfrac{kN}{a})&=&n_1\\&=&N(0, \frac{x}{a})
\end{eqnarray*}

and

\begin{eqnarray*}
 R(0, \frac{x}{a})-R(\dfrac{kN-x}{a}, \dfrac{kN}{a})&=& N(\dfrac{kN-x}{a}, \dfrac{kN}{a})-N(0, \dfrac{x}{a})\\ \\& = &0\\
\end{eqnarray*}

Case 3: Suppose $ p\nmid a$. Then $kN_p/p$ is an integer since $kN_p$ contains all primes but $p$. Let \[\frac{kN_p}{a}=n_2 \quad  \mbox{and} \quad \frac{x}{a}=n_1+r_1\]  for some integers $n_1, n_2$ and  $0< r_1<1$. Then

\begin{eqnarray*}
  N(\dfrac{kN-x}{a}, \dfrac{kN}{a}) &=& N(\dfrac{kN}{a}-\frac{x}{a}, \dfrac{kN}{a})\\
 &=&N (n_2-n_1-r_1, \quad n_2)
\end{eqnarray*}

The integers in the interval above are given by

\[\{n_2-n_1, \quad n_2-n_1+1, \quad n_2-n_1+2, \quad \cdots n_2-1 \}\]

The number of terms is

\[n_2-1-(n_2-n_1)+1=n_1\]

Thus

\begin{eqnarray*}
 N(\dfrac{kN-x}{a}, \dfrac{kN}{a})
 &=&n_1\\&=& N(0, \dfrac{x}{a})
\end{eqnarray*}

And
\begin{eqnarray*}
 R(0, \dfrac{x}{a})-R(\dfrac{kN-x}{a}, \dfrac{kN}{a})&= &N(\dfrac{kN-x}{a}, \dfrac{kN}{a})-N(0, \dfrac{x}{a})\\&=&0 \\
\end{eqnarray*}

\end{proof}

\section{The Main Theorem}

Let
\begin{equation}
B_k= \{a|m(a)<x, p/a,  R(0, \dfrac{ x}{a})> r(0, \dfrac{kN_p}{a})\}
\end{equation}

Then the lemma above implies

\begin{equation}
|S_x(\left(kN_p-x,\quad kN_p\right)|-|S_x\left(0,\quad  x\right)|=M(B_k) \label{lem3res}
\end{equation}

where $M$ is the Merten's function defined on the set $B_k$.
\bigskip

While  lemma (\ref{lem3}) is significant in that it transforms a complicated residual sum into a counting function, we don't know anything about the structure of $B_k$ explicitly,   It turns out, however, summing up (\ref{lem3res}) for all $k\leq p-1$ gives us exactly what we need.

Before we prove our main result, we need one more variant of $M$.  Let

\begin{eqnarray}
M_p^x(u, v)&=& \sum_{u<a<v, p/a, m(a)<x}\mu(a)
\end{eqnarray}

Namely, $M_p^x(u, v)$ is the set of square free integers in $(u, v)$ that are divisble by $p$ with no prime factors greater than $x$. 
\bigskip

\begin{thm}
Given a positive irrational number $x$, and a prime $p$ with $p<x$, and $N_p$ defined above, we have
\begin{eqnarray}\sum_{k=1}^{p-1}[|S\left(kN_p-x,\quad  kN_p\right)|-|S\left(0,\quad  x\right)|]
&=&-\sum_{n=1}^{m} \sum_{i=1}^{p-1} M'_{p}(\frac{x}{np},  \quad \frac{x}{np-i})\label{thmmain}\\
&=&([\log_px])(p-1)
\end{eqnarray}
\end{thm}
where $m=[x/p]+2$. Note that the intervals above will not contain any integer for $n>m$.

\begin{proof}

By the lemma (\ref{lem3}), the summation above is just a counting function for square free integer $a$'s ,with $p/a$ , $m(a)<x$, and the condition

 \begin{equation}
R(0, \frac{x}{a})> R(0, \frac{kN_p}{a}) \label{condr}
\end{equation}

for some $0<k\leq p-1$. If a square free integer $a$ satisfies the condition above for a specific $k$, it will contribute to the left side of (\ref{thmmain}) by $\mu(a)$.  Note that for any specific $a$, the condition (\ref{condr})  can be satisfied for more than one $k$. In that case, the number $\mu(a)$ has to be multiplied by the number of such $k$'s.\\

Now, if  $p/a$, then $a=pb$ for some integer $b$ with $p\nmid b$. Since $p\nmid b$, we have $b/N_p$. Let $N_p/b=M$ for some integer $M$.  Then we have

\begin{equation}
\frac{kN_p}{a}=\frac{kN_p}{bp}=k\frac{M}{p}=n_k+\frac{i_k}{p}
\end{equation}

 for some integers $n_k, i_k$ and $1\leq i_k\leq  p-1$. Therefore

\begin{equation}
R(0, \frac{kN_p}{a})=\frac{i}{p}
\end{equation}

for some $i=1,2, \cdots p-1$. Note that since $p$ is prime, $ k_1\neq k_2 \Rightarrow i_{k_1}\neq i_{k_2}$. Therefore the set \[\{R(0, \frac{kN_p}{a})| k=1, 2, \cdots p-1\}\] will be a permutation of \[\{\frac{i}{p}| i=1, 2, \cdots p-1\} \]

Consider the following disjoint set of intervals

\[\{(n-\frac{i-1}{p}, \quad  n-\frac{i}{p})| n=1, 2, \cdots =\infty \quad \mbox{and}\quad  i=1, 2\cdots p\}\] This set contains all positive irrational numbers, thus, it contains the fraction "$x/a$" for any irrational $x$, and any integer $a$. This implies that the set of intervals

\begin{equation}
(px, \quad \infty)\bigcup \{(\frac{x}{n-(i-1)/p}, \quad \frac{x}{n-i/p})|n=1, 2, \cdots =\infty \quad \mbox{and}\quad  i=1, 2\cdots p \quad (n, i)\neq (1, p)\} \label{interval}
\end{equation}

contains all integers since  \[n-(i-1)p< x/a< n-i/p\]implies
\[a\in (\frac{x}{n-(i-1)/p},  \quad \frac{x}{n-i/p})\]

 We will calculate the contribution of each square free integer  $a$ to (\ref{thmmain})  based on each of these intervals. As we will see, the square free integers in the same interval will contribute to the sum by the same exact amount.\\

Now for a given integer $n$, suppose $a\in (\frac{x}{n},  \frac{x}{n-1/p})$, a square free integer  with $p/a$ and $m(a)<x$. Then

\begin{eqnarray}
\frac{x}{n} < a < \frac{x}{n-1/p} &\Rightarrow & n-\frac{1}{p} <\frac{x}{a}< n \\
 &\Rightarrow&  \frac{p-1}{p} < R(0, \frac{x}{a})< 1
\end{eqnarray}
This implies $R(0, \frac{x}{b})> R(0, \frac{kN_p}{a})$ for ever value of $k$, since $R(0, \frac{kN_p}{a})=i/p$ for some integer $i=1, 2, \cdots p-1$. Thus such an $a$ will be counted  $p-1$ times in the sum (\ref{thmmain}), thus contribute by $(p-1)\mu(a)$ to the sum. The total contribution from $(\frac{x}{n},  \frac{x}{n-1/p})$ will therefore be  precisely

\begin{equation}
(p-1)M_p^x(\frac{x}{n}, \quad  \frac{x}{n-1/p})
\end{equation}

\bigskip
Similarly,
\begin{eqnarray*}
  a\in (\frac{x}{n-1/p},  \quad \frac{x}{n-2/p})  \quad &\Rightarrow& \quad n-\frac{2}{p} <\frac{x}{a}< n-\frac{1}{p}\\\\
 &\Rightarrow& \quad  \frac{p-2}{p} < R(0, \frac{x}{a})< \frac{p-1}{p}
\end{eqnarray*}

Thus  $R(0, \frac{x}{a})> R(0, \frac{kN_p}{a})$ for $p-2$ different values of $k$. Therefore the contribution from all integers in $ (\frac{x}{n-1/p},  \frac{x}{n-2/p})$ to the sum above is

\begin{equation}
(p-2)M_p^x(\frac{x}{n-1/p}, \quad  \frac{x}{n-2/p})
\end{equation}

\bigskip

Continuing this way, for a general $i<p-1$, we have
\begin{eqnarray*}
  a\in (\frac{x}{n-(i-1)/p},  \quad \frac{x}{n-i/p})  \quad &\Rightarrow& \quad n-\frac{i}{p} <\frac{x}{a}< n-\frac{i-1}{p}\\\\
 &\Rightarrow& \quad  \frac{p-i}{p} < R(0, \frac{x}{b})< \frac{p-i+1}{p} \\
\end{eqnarray*}

Thus  $R(0, \frac{x}{b})> R(0, \frac{kN_p}{a})$ for $p-i$ values of $k$. Therefore, the contribution from $ (\frac{x}{n-k/p},  \frac{x}{n-(k+1)/p})$ to the sum above is

\begin{equation}
(p-i)M_p^x(\frac{x}{n-(1-1)/p}, \quad  \frac{x}{n-i/p})
\end{equation}
\bigskip

Finally, for $n\neq 1$,
\begin{eqnarray*}
  a\in (\frac{x}{n-(p-1)/p},  \quad \frac{x}{n-1})  \quad &\Rightarrow& \quad n-1 <\frac{x}{a}< n-\frac{p-1}{p}\\\\
 &\Rightarrow& \quad  0< R(0, \frac{x}{a})< \frac{1}{p} \\
\end{eqnarray*}

Thus $ R(0, \frac{x}{b})< R(0, \frac{kN_p}{a})$ for every value of $k=1, 2..$.  Therefore the contribution from the intervals  $(\frac{x}{n-(p-1)/p},  \quad \frac{x}{n-1})$ will be zero.

Also,
\begin{eqnarray*}
  a\in (px,  \quad \infty)  \quad &\Rightarrow& \quad 0<\frac{x}{a}<\frac{1}{p} \\
 &\Rightarrow &  \quad  0< R(0, \frac{x}{a})< \frac{1}{p}\label{px} \\
\end{eqnarray*}

Thus the contribution from $(px,  \quad \infty)$ is also zero. This is especially important as it eliminates products greater than $px$ from consideration which is the overwhelming majority. \\

Adding up the contributions from all intervals in (\ref{interval}), which contain all square free integers in the summation (\ref{thmmain}),  we arrive at

\begin{equation}
\sum_{n=1}^m\sum_{i=1}^{p-1} (p-i)M_p^x(\frac{x}{n-(i-1)/p},  \quad \frac{x}{n-i/p})
\end{equation}

which is equal to

\[\sum_{n=1}^{m} \sum_{i=1}^{p-1} M_{p}^x(\frac{x}{n},  \quad \frac{x}{n-i/p})\]

since  \[\sum_{i=1}^{p-1} (p-i)M_p^x(\frac{x}{n-(i-1)/p},  \quad \frac{x}{n-i/p}) =\sum_{i=1}^{p-1} M_{p}^x(\frac{x}{n},  \quad \frac{x}{n-i/p})\]

\bigskip
Note that $a\in (\frac{x}{n},  \quad \frac{x}{n-i/p})$ implies $a< px$, which implies $a/p<x \Rightarrow m(a)<x$, thus, we can replace $M_p^x$ with $M_p$.

Also since \[M_p(x, y)=-M'_p(\frac{x}{p}, \frac{y}{p})\] we have

\begin{eqnarray*}M_p(\frac{x}{n},  \quad \frac{x}{n-i/p})&=&-M'_p(\frac{1}{p}\frac{x}{n},  \quad \frac{1}{p}
\frac{x}{n-i/p})\\&=&-M'_p(\frac{x}{np},  \quad
\frac{x}{np-i})
\end{eqnarray*}

Therefore,
\begin{eqnarray}\sum_{k=1}^{p-1}[|S\left(kN_p-x,\quad  kN_p\right)|-|S\left(0,\quad  x\right)|]
&=&-\sum_{n=1}^{m} \sum_{i=1}^{p-1} M^x_{p}(\frac{x}{n},  \quad \frac{x}{n-i/p})\\
&=&-\sum_{n=1}^{m} \sum_{i=1}^{p-1} M_{p}(\frac{x}{n},  \quad \frac{x}{n-i/p})\\
&=&-\sum_{n=1}^{m} \sum_{i=1}^{p-1} M'_{p}(\frac{x}{np},  \quad \frac{x}{np-i})\\
\end{eqnarray}

Since
\[|S\left(kN_p-x,\quad  kN_p\right)|-|S_x(0, \quad x)|=[\log_px]+1\] by (\ref{corollary1}), for every $k$, the theorem follows.

\end{proof}

\end{document}